\documentclass[12pt,reqno]{article}

\usepackage[usenames]{color}
\usepackage{amssymb}
\usepackage{amsmath}
\usepackage{amsthm}
\usepackage{amsfonts}
\usepackage{amscd}
\usepackage{graphicx}

\usepackage{mathtools}

\DeclarePairedDelimiter{\abs}   {\lvert } {\rvert }
\DeclareRobustCommand{\stirling}{\genfrac\{\}{0pt}{}}

\usepackage[colorlinks=true,
linkcolor=webgreen,
filecolor=webbrown,
citecolor=webgreen]{hyperref}

\definecolor{webgreen}{rgb}{0,.5,0}
\definecolor{webbrown}{rgb}{.6,0,0}

\usepackage{color}
\usepackage{fullpage}
\usepackage{float}

\usepackage{graphics}
\usepackage{latexsym}
\usepackage{epsf}

\newcommand{\seqnum}[1]{\href{https://oeis.org/#1}{\underline{#1}}}

\theoremstyle{plain}
\newtheorem{theorem}{Theorem}
\newtheorem{corollary}[theorem]{Corollary}
\newtheorem{lemma}[theorem]{Lemma}

\theoremstyle{definition}

\theoremstyle{remark}
\newtheorem{remark}[theorem]{Remark}

\numberwithin{equation}{section}

\title{On Colored Factorizations}
\author{Jacob Sprittulla \\ sprittulla@alice-dsl.de}
\date{\today}

\begin{document}
	
\begin{center}
	\vskip 1cm{\LARGE\bf  On Colored Factorizations\\
	\vskip 1cm}
	\large
	Jacob Sprittulla\\
	Germany\\
	\href{mailto:sprittulla@alice-dsl.de}{\tt sprittulla@alice-dsl.de}
\end{center}

\vskip .2 in
	
\begin{abstract}
	We study the number of factorizations of a positive integer, where the parts of the factorization are of $l$ different colors (or kinds). Recursive or explicit formulas are derived for the case of unordered and ordered, distinct and non-distinct factorizations with at most and exactly $l$ colors. 
\end{abstract}

\section{Introduction}
	For $l \geq 1$ and $n \geq 2$ we denote the following factorizitation counting functions, where the parts of the factorization are always $\geq 2$:  
	
	\begin{itemize}
		\item $A_l(n)$ denotes the number of \textit{unordered} factorizations into parts of \textit{at most} $l$ different colors (or into parts of at most $l$ different kinds),
		\item $B_l(n)$ denotes the number of \textit{unordered} factorizations into \textit{distinct} parts of \textit{at most} $l$ colors (more precisely, parts are considered to be distinct, if they have the same color but a different value or if they have a different color),
		\item $\tilde{A}_l(n)$ denotes the number of \textit{ordered} factorizations into parts of \textit{at most} $l$ colors.
		\item $\tilde{B}_l(n)$ denotes the number of \textit{ordered} factorizations into \textit{distinct} parts of \textit{at most} $l$ colors.
	\end{itemize}
    
    We will denote the corresponding factorization counting functions into parts with \textit{exactly} $l$ colors by lowercases. For example $a_l(n)$ denotes the number of \textit{unordered} factorizations into parts of \textit{exactly} $l$ colors, and analogously for $b_l(n)$, $\tilde{a}_l(n)$ and $\tilde{b}_l(n)$. 
	
	We will sometimes also briefly refer to $A_l(n)$ as \textit{the number of $l$-colored factorizations} and analogously for the other factorization counting functions.
	
	We give an example. Let $n=12=(2^2)(3)$ and $l=2$. We list all factorizations and the values of the factorization counting functions below.
	\begin{align*}
		A_2(12) = 16 = \# \{ &(12),(\bar{12}),  (6,2), (6,\bar{2}), 
		(\bar{6},2), (\bar{6},\bar{2}), (4,3), (4,\bar{3}), (\bar{4},3), (\bar{4},\bar{3}),    \\
		&(3,2,2), (\bar{3},\bar{2},\bar{2}), (3,2,\bar{2}), (\bar{3},2,\bar{2}), (3,\bar{2},\bar{2}), (\bar{3},2,2) \} \text{,} \\
		B_2(12) = 12 = \#\{&(12), (\bar{12}), (6,2), (6,\bar{2}),
		(\bar{6},2), (\bar{6},\bar{2}),  (4,3),(4,\bar{3}), (\bar{4},3), (\bar{4},\bar{3}),    \\
		& (3,2,\bar{2}),(\bar{3},2,\bar{2})\} \text{,} \\ 
		a_2(12) =8 = \# \{ &(6, \bar{2}), (\bar{6},2), (4,\bar{3}),
		(\bar{4},3), (3,2,\bar{2}), (\bar{3},2,\bar{2}), (3,\bar{2},\bar{2}), 
		(\bar{3},2,2) \} \text{,} \\
		b_2(12) =6 = \#\{&(6,\bar{2}),(\bar{6},2),  (4,\bar{3}),
		(\bar{4},3), (3,2,\bar{2}), (\bar{3},2,\bar{2})\} \text{,} \\  
		\tilde{a}_2(12)  =26 = \#\{& (6,\bar{2}), (\bar{2},6),
		(\bar{6},2), (2,\bar{6}), (4,\bar{3}), (\bar{3},4), (\bar{4},3), (3,\bar{4}),  \\ 
		&(3,2,\bar{2}),(2,3,\bar{2}),(2,\bar{2},3),(3,\bar{2},2),(\bar{2},3,2),(\bar{2},2,3), \\
		&(\bar{3},2,\bar{2}), (2,\bar{3},\bar{2}), (2,\bar{2},\bar{3}), (\bar{3},\bar{2},2), (\bar{2},\bar{3},2),(\bar{2},2,\bar{3}),   	 \\
		&(\bar{3},2,2), (2,\bar{3},2), (2,2,\bar{3}), 
		(\bar{3},\bar{2},2), (\bar{3},2,\bar{2}), 
		(2,\bar{3},\bar{3}) \} \text{,} \\
		\tilde{b}_2(12)  =20 = \#\{& (6,\bar{2}), (\bar{2},6),
		(\bar{6},2), (2,\bar{6}), (4,\bar{3}), (\bar{3},4), (\bar{4},3), (3,\bar{4}),  \\ 
		&(3,2,\bar{2}),(2,3,\bar{2}),(2,\bar{2},3),(3,\bar{2},2),(\bar{2},3,2),(\bar{2},2,3), \\
		&(\bar{3},2,\bar{2}), (2,\bar{3},\bar{2}), (2,\bar{2},\bar{3}), (\bar{3},\bar{2},2), (\bar{2},\bar{3},2),(\bar{2},2,\bar{3}) \}  
		\text{.}
	\end{align*}
	Here $\#$ denotes the number of elements of a set and $(d,\bar{d})$ denote the two colors of an integer $d$. For completeness, we also note that $\tilde{A}_2(12) = 42$ and $\tilde{B}_2(12) = 30$. For small values of $l$, some of these sequences can be found in the \textit{On-Line Encyclopedia of Integer Sequences} (OEIS) \cite{OEIS}, for example \seqnum{A301830} ($(A_2(n))_{n \geq 1}$) and \seqnum{A328706} ($(B_2(n))_{n \geq 1}$).
	
	For $l=1$, we get the standard (uncolored) factorization counting functions, which we will denote by $f(n):=A_1(n)$, $g(n):=B_1(n)$, $\tilde{f}(n):=\tilde{A}_1(n)$ and $\tilde{g}(n):=\tilde{B}_1(n)$. 	
	
	It is easy to see that all the above functions are \textit{prime independent}, meaning that their value is completely determined by the prime signature of $n$. For example, we have $A_2(12)=A_2(75)=16$, since $12=(2^2)(3)$ and $75=(3)(5^2)$ share the same prime signature $(2,1)$. 
	
	We will denote primes by $\pi$ and $\pi_i$ ($i=1,2,\dots)$. Evaluated at prime powers $\pi^n$, we get the corresponding partition or composition functions. For example $A_l(\pi^n)$ is the number of partitions of $n$ into parts of at most $l$ kinds, $B_l(\pi^n)$ the number of compositions of $n$ into parts of at most $l$ kinds. For small values of $l$, these sequences can be found in the OEIS, for example \seqnum{A000712} ($(A_2(\pi^n))_{n \geq 1}$), \seqnum{A022567} ($(B_2(\pi^n))_{n \geq 1}$), \seqnum{A025192} ($(\tilde{A}_2(\pi^n))_{n \geq 1}$) and \seqnum{A032005} ($(\tilde{B}_2(\pi^n))_{n \geq 1}$).
	
	The number of $l$-colored factorizations doesn't seem to have attracted much attention in the literature yet. The only reference we found is the paper of Subbarao \cite{Sub04}, where the Dirichlet generating function (dgf) of $A_l(n)$ is given. The aim of this paper is to study the recursive structure and the interdependence of these functions. We will give recursive or explicit formulas and the dgf's for all variants mentioned. Further, we derive an average order for the number of $l$-colored ordered factorizations (Corollary \ref{co:avor}).

\section{Unordered colored factorizations}
	In the case of $l$-colored unordered factorizations, the recursive structure of the factorization counting functions can be deduced with standard methods from their dgf's.  
	Subbarao \cite[Equation 2.4]{Sub04} noted that the dgf of $A_l(n)$ is given by the $l$-th power of the dgf of $f(n)$. We will motivate and generalize this result in the next lemma. 
	
	For an arithmetic function $\phi(n)$, we denote by $\phi^{*l}(n)$ the \textit{$l$-fold Dirichlet convolution} of $\phi$ with itself, i.e. $\phi^{*l}=\phi^{*(l-1)}*\phi$, for $l \geq 2$ and $\phi^{*1}=\phi$, where $*$ denotes Dirichlet convolution. We will also say that $\phi^{*l}$ is the \textit{$l$-th Dirichlet power} of $\phi$. 
	
	Further, we denote by $f_{k,l}(n)$ ($g_{k,l}(n)$) the number of (distinct) unordered $l$-colored factorizations with exactly $k$ parts. These auxiliary functions are needed to deduce formulae for $\tilde{B}_l(n)$ (Theorem \ref{th:tBln} below) and $\tilde{b}_l(n)$. For $l=1$, $f_k(n):=f_{k,1}(n)$ ($g_k(n):=g_{k,1}(n))$ denotes the number of unordered (distinct) factorizations of $n$ with exactly $k$ parts. The dgfs of these sequences are given by
	\begin{align}
		\label{eq:fkdgf}
		1+\sum_{n=2}^\infty \sum_{k=1}^\infty f_{k}(n) n^{-s} z^k
		&= \prod_{n=2}^\infty (1 - z n^{-s})^{-1} \\
		\label{eq:gkdgf}
		1+\sum_{n=2}^\infty \sum_{k=1}^\infty g_{k}(n) n^{-s} z^k
		&= \prod_{n=2}^\infty (1 + z n^{-s}) 
		\text{,}
	\end{align}	
	see Hensley \cite[Equation 1.4]{Hen87} and Subbarao \cite[Equation 2.4]{Sub04}. We denote complex numbers by $s$ and $z$. 
	
	\begin{lemma}
		\label{lm:ufcdgf}
		For an integer $l \geq 1$ and $\operatorname{Re}(s) >1$, we have $A_l=f^{*l}$ and $B_l=g^{*l}$ and therefore
		\begin{align}
			\label{eq:Aldgf}
			\mathcal{A}_l(s) &:= 1 + \sum_{n=2}^\infty A_l(n) n^{-s} 
			=\prod_{n=2}^\infty (1 - n^{-s})^{-l}  \\
			\label{eq:Bldgf}
			\mathcal{B}_l(s) &:= 1 + \sum_{n=2}^\infty B_l(n) n^{-s}
			= \prod_{n=2}^\infty (1 + n^{-s})^l 
			\text{.}
		\end{align}
		Further, for $k,l \geq 1$ and $\operatorname{Re}(s),\operatorname{Re}(z) >1$ we have 
		\begin{align}
			\label{eq:fkldgf}
			\mathcal{F}_{l}(s) 
			&:= 1 + \sum_{n=2}^\infty \sum_{k=1}^\infty f_{k,l}(n) n^{-s} z^k
			= \prod_{n=2}^\infty (1 - z n^{-s})^{-l} \\
			\label{eq:gkldgf}
			\mathcal{G}_{l}(s) 
			&:= 1 + \sum_{n=2}^\infty \sum_{k=1}^\infty g_{k,l}(n) n^{-s} z^k
			= \prod_{n=2}^\infty (1 + z n^{-s})^{l} 
			\text{.}
		\end{align}
	\end{lemma}
	\begin{proof}
		Let $n \geq 2$ be given. Any unordered $l$-colored factorization of $n$ can be constructed of all pairs if integers $n_1$ and $n_2$ with $n_1 n_2 = n$, by considering factorizations of the first $(l-1)$ colors of $n_1$ and factorizations of the $l$-th color of $n_2$; hence $A_l(n)=\sum_{n_1 n_2=n} A_{l-1}(n_1) f(n_2)=(A_{l-1} * f)(n)$. This shows \eqref{eq:Aldgf} by induction on $l$; the proof of \eqref{eq:Bldgf} is similar.
		
		Analogously, all factorizations with exactly $k$ parts of at most $l$ colors can be constructed of all pairs of integers $n_1 n_2=n$ and $k_1+k_2=k$, where $n_1$ is made of $k_1$ parts of the first $l-1$ colors and $n_2$ is made of $k_2$ parts of the last $l$-th color. Therefore 
		\begin{align*}
			\sum_{n=2}^\infty \sum_{k=1}^\infty f_{k,l}(n) n^{-s} z^k
			&= \sum_{n=2}^\infty \sum_{k=1}^\infty \left( \sum_{n_1 n_2=n} 
			\sum_{k_1+k_2=k} f_{k_1,l-1}(n_1) f_{k_2,1}(n_2) \right)
			n^{-s} z^k \\
			&=\sum_{n_1,n_2=2}^\infty \sum_{k_1,k_2=1}^\infty 
			f_{k_1,l-1}(n_1) f_{k_2,1}(n_2) 
			{n_1}^{-s} z^{k_1} {n_2}^{-s} z^{k_2} \\
			&=\left( \sum_{n=2}^\infty \sum_{k=1}^\infty 
			f_{k,l-1}(n) n^{-s} z^k \right)
			\left( \sum_{n=1}^\infty \sum_{k=1}^\infty 
			f_{k,1}(n) n^{-s} z^k \right)
			\text{.}
		\end{align*}	
		It follows by induction on $l$ that $\mathcal{F}_{k,l}(s)$ is the $l$-th power of the the dgf of $f_{k}(n)$. Together with \eqref{eq:fkdgf}, this shows \eqref{eq:fkldgf}. The proof of \eqref{eq:gkldgf} is similar, using \eqref{eq:gkdgf}.
    \end{proof}
	
	\begin{remark}
		We can change the set admissible parts $\mathcal{S}=\{2,3,4,5,\dots\}$ in Lemma \ref{lm:ufcdgf} to any other nonempty subset of $\mathcal{S}$ by changing the range of the products of the dgfs accordingly. For example, if we put $\mathcal{S}=\{2,3,5,7\dots\}$, the set of all primes, we get the dgfs of the number of $l$-colored \textit{prime factorizations} of $n$.
		By the Euler product of the Riemann zeta function, we then get that this dgf equals $\zeta(s)^l$. But $\zeta(s)^l$ is also the dgf of the number of ordered factorizations into $l$ parts $\geq 1$ (the $l$-th divisor function, denoted by $d_l(n)$). Therefore, we can conclude that, for every $n \geq 2$, the number of $l$-colored unordered prime-factorizations equals the number of ordered factorizations with $l$ parts $\geq 1$, since both sequences share the same dgf $\zeta(s)^l$.
		A similar argument shows that the number of ordered factorizations into $l$ parts $\geq 2$ (denoted by $f_l(n)$) equals the number of exactly $l$-colored unordered prime factorizations, since both sequences share the same dgf $(\zeta(s)-1)^l$ (see Remark \ref{rm:dgfexact} below). For example, for $n=12$ and $l=2$, we have
		\begin{align*}
			d_2(12) = 6 
			&= \# \{ (12,1), (1,12), (6,2), (2,6), (4,3), (3,4) \}  \\
			&= \# \{ (3,2,2), (\bar{3},2,2), (3,\bar{2},2), 
			(\bar{3},\bar{2},2), (3,\bar{2},\bar{2}), 
			(\bar{3},\bar{2},\bar{2}) \},  \\
			f_2(12) = 4 
			&= \# \{ (6,2), (2,6), (4,3), (3,4) \}  \\
			&= \# \{ (\bar{3},2,2), (3,\bar{2},2), 
			(\bar{3},\bar{2},2), (3,\bar{2},\bar{2}) \} 
			\text{.}
		\end{align*}
	\end{remark}

	\begin{remark}
		\label{rm:algolist}
		An effective algorithm to produce a list of all unordered factorizations for a given $n \geq 2$ can be found in Knopfmacher and Mays \cite{Kno06}. As explained in the proof of Lemma \ref{lm:ufcdgf}, Equation \eqref{eq:Aldgf}, all $l$-colored factorizations of an integer $n$ can be constructed from $(l-1)$-colored factorizations of the divisors $d$ of $n$ and the factorizations of $n/d$. This can be used to produce recursively a list of all unordered $l$-colored factorizations based on a list of all unordered factorizations. This algorithm can be extended to produce lists of factorizations for the other factorization counting functions analyzed in this paper.
	\end{remark}

	We can now derive a recursive equations for $A_l(n)$, $B_l(n)$, $f_{k,l}(n)$ and $g_{k,l}(n)$, by making use of the dgf's.
	\begin{theorem}
		\label{th:ufcrec}
		For $n \geq 2$, $k \geq 1$ and an integer $l \neq 0$, we have 
		\begin{align}
			\label{eq:recA}
			A_l(n) \log n
			&= l \sum_{\substack{ d^i \lvert n \\ d \geq 2  }}
			A_l(n/d^i) \log d  \\
			\label{eq:recB}
			B_l(n) \log n
			&= l \sum_{\substack{ d^i \lvert n \\ d \geq 2  }}
			(-1)^{i+1} B_l(n/d^i) \log d  \\
			\label{eq:recfkl}
			f_{k,l}(n) \log n
			&= l \sum_{\substack{ d^i \lvert n \\ d \geq 2, k \geq i}}
			f_{k-i,l}(n/d^i) \log d \\
			\label{eq:recgkl}
			g_{k,l}(n) \log n
			&= l \sum_{\substack{ d^i \lvert n \\ d \geq 2, k \geq i}}
			(-1)^{i+1} g_{k-i,l}(n/d^i) \log d 
			\text{,}
		\end{align}
		with boundary conditions $A_l(1)=B_l(1)=1$ and 
		\begin{align*}
			f_{k,l}(1) = g_{k,l}(1) &= 
			\begin{cases}
			1, & \text{if $k=0$;} \\
			0, & \text{if $k \geq 1$.} \\
		\end{cases} 	
		\end{align*}
	\end{theorem}
	\begin{proof}
		Fix an integer $l \neq 0$. We proof \eqref{eq:recB} by using logarithmic differentiation of the dgf and the identity 
		$\log(1+x) = \sum_{i=1}^\infty \tfrac{1}{i} (-1)^{i+1} x^i$
		(for $\abs x < 1$). For positive integers $i$ and $n$, we denote by $\lambda_i(n)$ the indicator function of the $i$-th power, i.e.
		\begin{align*}
			\lambda_i(n) &:= 
			\begin{cases}
				1, & \text{if $n=d^i$, for some integer $d$;}\\
				0, & \text{else.}
			\end{cases} 
		\end{align*}
		We have
		\begin{align*}
			\tfrac{\partial}{\partial s}  \log \mathcal{B}_l(s)
			&= \tfrac{\partial}{\partial s} \left(
			l \sum_{n=2}^\infty \log(1 + n^{-s}) \right) \\
			&= \tfrac{\partial}{\partial s} \left(
			l \sum_{n=2}^\infty \sum_{i=1}^\infty 
			\tfrac{1}{i} (-1)^{i+1} n^{-is} \right) \\
			&= -l \sum_{n=2}^\infty \sum_{i=1}^\infty 
			(-1)^{i+1} n^{-is} \log n  \\
			&= -l \sum_{n=2}^\infty \sum_{i=1}^\infty 
			(-1)^{i+1} \lambda_i(n) n^{-s} \log(n^{1/i})
			\text{.}
		\end{align*}
		By the definition of $\mathcal{B}_l(s)$, we also have
		\begin{align*}
			\tfrac{\partial}{\partial s}  \log \mathcal{B}_l(s)
			&= \frac {\tfrac{\partial}{\partial s} \mathcal{B}_l(s)}{ \mathcal{B}_l(s)  }  \\
			&= \frac{-1}{\mathcal{B}_l(s)}
			\sum_{n=2}^\infty B_l(n)  n^{-s} \log n
			\text{.}
		\end{align*}
		Equating both equations for the logarithmic derivative, we get
		\begin{align*}
			\sum_{n=2}^\infty B_l(n)  n^{-s} \log n
			&= l \left( \sum_{n=2}^\infty B_l(n)  n^{-s} \right)
			\left( \sum_{n=2}^\infty \sum_{i=1}^\infty 
			(-1)^{i+1} \lambda_i(n) n^{-s} \log(n^{1/i}) \right)  \\
			&= l \sum_{n=2}^\infty n^{-s} 
			\sum_{\substack{ d \lvert n \\ d \geq 2  }} 
			(-1)^{i+1}  B_l(n/d) \lambda_i(d) \log(d^{1/i})  \\
			&= l \sum_{n=2}^\infty n^{-s} 
			\sum_{\substack{ d^i \lvert n \\ d \geq 2  }} 
			(-1)^{i+1} B_l(n/d^i) \log d
			\text{.}
		\end{align*}
		Extracting coefficients on both sides of this equation, we get \eqref{eq:recB}. Analogously, by using the identity 
		$\log(1-x) = -\sum_{i=1}^\infty \tfrac{1}{i} x^i$
		(for $\abs x < 1$), we get \eqref{eq:recA}.

		Next, we prove \eqref{eq:recgkl}. As in the proof of \eqref{eq:recB}, we get
		\begin{align*}
			\tfrac{\partial}{\partial s}  \log \mathcal{G}_{l}(s)
			&= \tfrac{\partial}{\partial s} \left(
			l \sum_{n=2}^\infty \log(1 + z n^{-s}) \right) \\
			&= -l \sum_{n=2}^\infty \sum_{i=1}^\infty 
			(-1)^{i+1} z^i \lambda_i(n) n^{-s} \log(n^{1/i})
			\text{}
		\end{align*}
		and
		\begin{align*}
			\tfrac{\partial}{\partial s}  \log \mathcal{G}_l(s) 
			&= \frac{-1}{\mathcal{G}_l(s)}
			\sum_{k=1}^\infty z^k  \sum_{n=2}^\infty  g_{k,l}(n)  n^{-s} \log n
			\text{.}
		\end{align*}
		Equating these two expressions, we get (by absolute convergence)
		\begin{align*}
			\sum_{n=2}^\infty  \sum_{k=1}^\infty g_{k,l}(n)  n^{-s} z^k  \log n
			&= \left(  \sum_{n=2}^\infty \sum_{k=1}^\infty 
			g_{k,l}(n) n^{-s} z^k \right)
			\left( l \sum_{n=2}^\infty \sum_{i=1}^\infty 
			(-1)^{i+1} z^i \lambda_i(n) n^{-s} \log(n^{1/i})  \right) \\
			&= l \sum_{k=1}^\infty \sum_{i=1}^\infty z^k z^i
			\left(  \sum_{n=2}^\infty n^{-s}
			\sum_{\substack{ d \lvert n \\ d \geq 2  }} 
			(-1)^{i+1}  g_{k,l}(n/d) \lambda_i(d) \log(d^{1/i})  \right) \\			
			&= l \sum_{k=1}^\infty  \sum_{n=2}^\infty n^{-s} z^k 
			\sum_{\substack{ d \lvert n \\ d \geq 2  }} \sum_{i=1}^k
			(-1)^{i+1}  g_{k-i,l}(n/d) \lambda_i(d) \log(d^{1/i})  
			\text{.}
		\end{align*}	
		This shows \eqref{eq:recgkl} by extracting coefficients. The proof of \eqref{eq:recfkl} is similar.
	\end{proof}
	
	\begin{remark}
		The parameter $l$ in Theorem \ref{th:ufcrec} is not restricted to positive integers. For $l=-1$, we get recursive equations for the \textit{generalized M\"{o}bius functions} $\mu_f(n)$ and $\mu_g(n)$ defined by
		\begin{align}
			\mathcal{B}_{-1}(s)
			&=1+\sum_{n=2}^{\infty} \mu_f(n) n^{-s}
			= \prod_{n=2}^{\infty} ( 1+n^{-s})^{-1} \\
			\mathcal{A}_{-1}(s) 
			&=1+\sum_{n=2}^{\infty} \mu_g(n) n^{-s}
			= \prod_{n=2}^{\infty} ( 1-n^{-s}) 
			\text{.}
		\end{align}
		The function $\mu_f(n)$ ($\mu_g(n)$) counts the surplus of unordered (distinct) factorizations with an even number of parts over the unordered (distinct) factorizations with an odd number of parts, see Chamberland et al. \cite[Chapter 5]{Cha13} and Subbarao \cite[Section 6.3]{Sub04}.  
		The term \textit{generalized M\"{o}bius functions} is motivated by the observation that the M\"{o}bius function 
		\begin{align*}
			\mu(n) &= 
			\begin{cases}
				(-1)^{\omega(n)}, & \text{if $n$ is squarefree;}\\
				0, & \text{otherwise,}
			\end{cases} 
		\end{align*}	
		counts the surplus of the number of ordered factorizations with an even number of parts over the number of ordered factorizations with an odd number of parts, see Friedlander and Iwaniec \cite[Equation 17.2]{Fri10}. We denote by $\omega(n)$ the number of distinct prime factors of $n$.
    \end{remark}

	If $n$ is square free, we can derive an explicit formula for the number of $l$-colored unordered (distinct) factorizations of $n$. We denote by $\stirling{n}{k}$ the Stirling numbers of the second kind.
	\begin{theorem}
		Let $n=\prod_{i=1}^m \pi_i$ be a product of distinct primes and $l \geq 1$. Then
		\begin{align*}
			A_l(n) = B_l(n) =\sum_{k=1}^m l^k \stirling{m}{k} 
			\text{.}
		\end{align*}
	\end{theorem}
	\begin{proof}
		Let an integer $m \geq 1$ be given. For an integer $1 \leq k \leq m$, the number of unordered (distinct) factorizations of $n$ with exactly $k$ parts is given by $\stirling{m}{k}$. Denote the parts of such a factorization by $d_1, \dots, d_k$. By the structure of $n$, we can assume that $d_1 > \cdots > d_k$. Every $l$-colored factorization of $n$ with parts $d_1, \dots, d_k$ can be constructed by assigning $\alpha_i$ of the parts to the $i$-th color ($i=1,\dots,l$), with $\alpha_1 + \dots + \alpha_l = k$ and $\alpha_i \geq 0$. This shows that for every such partition $\alpha_1,\dots,\alpha_l$ there are $\binom{m}{\alpha_1 \cdots \alpha_l}$ $l$-colored factorizations. By the multinomial theorem, we have
		\begin{align*}
			\sum_{\alpha_1 + \dots + \alpha_l = k} \binom{m}{\alpha_1 \cdots \alpha_l} = l^k
			\text{.}
		\end{align*}
	    This shows that the number of $l$-colored unordered (distinct) factorizations with $k$ parts is given by $l^k \stirling{m}{k}$. The claim now follows by summing up over $k$.
	\end{proof}
	Let $Q_m$ denote the $m$-th primorial. For small values of $l$, the sequences $(A_l(Q_m))_{m \geq 1}$ can be found in the OEIS, e.g. \seqnum{A001861} ($l=2$) and \seqnum{A027710} ($l=3$).

\section{Ordered colored factorizations}
	For $l$-colored ordered factorizations, a simple combinatorial argument shows that the value of $\tilde{A}_l(n)$ can be derived from the number of factorizations with exactly $k$ different parts, denoted by $\tilde{f}_k(n)$. Recall that for $n=\prod_{j=1}^{\omega} \pi_j^{e_j}$ and $k \geq 1$, we have
	\begin{align}
		\label{eq:dgftf}
		\tilde{\mathcal{F}_k} (s) &:= 
		\sum_{n=1}^\infty \tilde{f}_k(n) n^{-s} = (\zeta(s)-1)^k  
	\end{align}
	and
	\begin{align}
		\label{eq:fortf}
		\tilde{f}_k(n) &= \sum_{i=0}^{k-1} (-1)^{i} \binom{k}{i}  
		\prod_{j=1}^{\omega} \binom{e_j + k-i-1}{e_j}
		\text{,}
	\end{align}
	see Knopfmacher and Mays \cite[Chapter 2.1]{Kno05}. We will denote by $\Omega=\Omega(n)$ the number of prime factors of $n$, counted with multiplicity.
	
	\begin{theorem}
		For $n \geq 2$ and $l \geq 1$, we have
		\begin{align}
			\label{eq:dgftA}
			\tilde{\mathcal{A}_l} (s) &:= 
			\sum_{n=1}^\infty \tilde{A}_l(n) n^{-s} = 
			\frac{1}{(l+1) - l \zeta(s)} 
			\text{}
		\end{align}
		and
		\begin{align}
			\label{eq:fortA}	
			\tilde{A}_l(n) &= \sum_{k=1}^\Omega l^k \tilde{f}_k(n) 
			\text{.}
		\end{align}
	\end{theorem}
	\begin{proof}
		Fix $n \geq 2$ and $l \geq 1$. Let $d_1 \cdots d_k$ be an ordered factorization of $n$. Every ordered factorization of $n$ with $k$ parts of $l$ kinds has a representation as $d_{1,i_1} \cdots d_{k,i_k}$ ($i_j \in \{1,2,\dots,l\}$ for $j=1,\dots,k$), where the second subscript indicates the kind of the part. This shows \eqref{eq:fortA}, by summing up over $k$.
		
		From \eqref{eq:dgftf},  \eqref{eq:fortA} and the formula for the geometric series we get 
		\begin{align*}
			\tilde{\mathcal{A}_l}(s)
			&=\sum_{k=1}^\infty l^k (\zeta(s)-1)^k 
			=\frac{1}{1-l(\zeta(s)-1)}  
			\text{.}
		\end{align*}
		This shows \eqref{eq:dgftA}.
		 
	\end{proof}	

	\begin{remark}
		It follows from \eqref{eq:dgftA} that the sequence $(\tilde{A}_l(n))_{n\geq 1}$ is the Dirichlet inverse of the sequence $(c_l(n))_{n\geq 1}$ with
		\begin{align*}
			c_l(n) &:= 
			\begin{cases}
			1, & \text{if $n=1$;}\\
			-l, & \text{for $n \geq 2$.}
		\end{cases} 	
		\end{align*} 
	\end{remark}

	\begin{remark}
		Let $Q_m$ denote the $m$-th primorial ($m \geq 1$). Then, we have 
		$\tilde{f}_k(Q_m) = k! \stirling{m}{k}$ 
		and therefore
		\begin{align*}
			\tilde{A}_l(Q_m) = \tilde{B}_l(Q_m) = \sum_{k=1}^m l^k k! \stirling{m}{k}
			\text{.}
		\end{align*}
		The same holds for any other square free number composed of $m$ distinct primes. For small values of $l$, the sequences 
		$(\tilde{A}_l(Q_m))_{m \geq 1}$ 
		can be found in the OEIS, e.g. \seqnum{A004123} ($l=2$) and \seqnum{A032033} ($l=3$).
	\end{remark}

	We can now derive an average order of $\tilde{A}_l(n)$ from its dgf \eqref{eq:dgftA} as a generalization of the formula for the average order of $\tilde{f}(n)$ of Kalm\'{a}r \cite{Kal30}. For real $x > 1$, we denote by $\zeta^{-1}(x)$ the unique real solution of the equation $\zeta(s)=x$.
	\begin{corollary}
		\label{co:avor}
		For an integer $l \geq 1$ we have
		\begin{align*}
			\tfrac{1}{x} \sum_{n \leq x} \tilde{A}_l(n) 
			&\sim \alpha_l x^{\beta_l-1} 
			\quad (x \rightarrow \infty) 
			\quad \text{with} \quad
			\beta_l = \zeta^{-1}(\tfrac{l+1}{l}) 
			\quad \text{and} \quad 
			\alpha_l = -\frac{1}{l \beta_l \zeta'(\beta_l)}
			\text{.}
		\end{align*}
	\end{corollary}
	\begin{proof}
		For $\operatorname{Re}(s) >1$ and $l \geq 1$, we get 
		\begin{align*}
			\mathcal{F}_l(s)
			&:=(l+1)-l \zeta(s)  \\
			&=(l+1) - l \sum_{i=0}^{\infty} \tfrac{1}{i!} \zeta^{(i)}(\beta_l) (s-\beta_l)^i  \\
			&=- l \sum_{i=1}^{\infty} \tfrac{1}{i!} \zeta^{(i)}(\beta_l) (s-\beta_l)^i \\
			&=-l (s-\beta_l) \sum_{i=1}^{\infty} \tfrac{1}{i!} \zeta^{(i)}(\beta_l) (s-\beta_l)^{i-1}
			\text{.}
		\end{align*}
		This shows that $\mathcal{F}_l(s)$ has a simple zero at $s=\beta_l$. Therefore, by \eqref{eq:dgftA}, the dgf $\tilde{\mathcal{A}_l}(s)=1/\mathcal{F}_l(s)$ of $\tilde{A}_l(n)$ has a simple pole at $\beta_l$ with 
		\begin{align*}
			\underset{s=\beta_l}{\operatorname{Res}} \tilde{\mathcal{A}_l} (s)
			&=\frac{1}{\mathcal{F}'_l(\beta_l)}=- \frac{1}{l \zeta'(\beta_l)} 
			\text{.}
		\end{align*}
	    The claim now follows from the Wiener-Ikehara Theorem \cite{Wie32}.
	\end{proof}

	In the next theorem, we derive a formula for the number of $l$-colored ordered distinct factorizations. 
	\begin{theorem}
		\label{th:tBln}
		For integers $n \geq 2$ and $l \geq 1$, we have
		\begin{align}
		\label{eq:tBln}
		\tilde{B}_l(n) = \sum_{k=1}^\Omega k! g_{k,l}(n)
		\text{,}
		\end{align}
		where $g_{k,l}(n)$ is defined recursively by \eqref{eq:recgkl}.	
	\end{theorem}
	\begin{proof}
		For $1 \leq k \leq \Omega$, let $d_{1,i_1} \cdots d_{k,i_k}$ be an $l$-colored unordered distinct factorization of $n$, where $i_j \in \{1,\dots,l\}$ ($j=1,\dots,k)$, are indicating the $l$-colors. Since the order of the parts is irrelevant, we can assume that $d_{1,i_1} > \cdots > d_{k,i_k}$. By changing the order of the parts, we can assign $k!$ ordered $l$-colored factorizations to each such factorization. Summing up over $k$ gives \eqref{eq:tBln}.
	\end{proof}

\section{The number of exactly $l$-colored factorizations}
	In the next theorem, we show that the number of factorizations of \textit{exactly} $l$ colors can easily be derived from the number of factorizations of \textit{at most} $l$ colors. 
	
	We say that the sequence $(\Phi_l)_{l \geq 0}$ is the \textit{binomial transform} of the sequence $(\phi_l)_{l \geq 0}$, if 
	$\Phi_l = \sum_{i=0}^{l} \binom{l}{i} \phi_l$ ($l \geq 0$) holds. It then follows that 
	$\phi_l = \sum_{i=0}^{l} (-1)^{l-i} \binom{l}{i} \Phi_l$ ($l \geq 0$), see Sloane and Plouffe \cite[Chapter 2.7]{Slo95}. We set $A_0(n)=a_0(n)=B_0(n)=b_0(n)=\tilde{A}_0(n)=\tilde{a}_0(n)=\tilde{B}_0(n)=\tilde{b}_0(n)=0$ for $n \geq 1$.

	\begin{theorem}
		\label{th:ufcbin}
		$A_l(n)$ is the binomial transform of $a_l(n)$, $B_l(n)$ is the binomial transform of $b_l(n)$, $\tilde{A}_l(n)$ is the binomial transform of $\tilde{a}_l(n)$ and $\tilde{B}_l(n)$ is the binomial transform of $\tilde{b}_l(n)$.
	\end{theorem}
	\begin{proof}
		We proof the assertion for $(A_l,a_l)$. Let  $n \geq 2$ and $l \geq 1$ be given.
		Among the $A_l(n)$ unordered factorizations of at most $l$ colors, there are exactly $\binom{l}{i} a_i(n)$ unordered factorizations of exactly $i$ of the $l$ colors, for $i=1,\dots,l-1$. Therefore, there are
		$a_l(n) = A_l(n) - \sum_{i=1}^{l-1} \binom{l}{i} a_i(n)$ unordered factorizations of exactly $l$ colors, so that
		\begin{align*}
			A_l(n) &= \sum_{i=0}^{l} \binom{l}{i} a_i(n)   
			\text{.}
		\end{align*}
		The same argument applies for unordered distinct factorizations, ordered factorizations and ordered distinct factorizations of parts with $l$ colors.
	\end{proof}
	
	\begin{remark}
		\label{rm:dgfexact}
		By Theorem \ref{th:ufcbin}, \eqref{eq:Aldgf} and the binomial theorem, we have
		\begin{align*}
			\sum_{n=1}^\infty a_l(n) n^{-s} 
			&= \sum_{n=1}^\infty n^{-s} 
			   \sum_{i=0}^{l} (-1)^{l-i}  \binom{l}{i} A_i(n)  \\
			&= \sum_{i=0}^{l} (-1)^{l-i}  \binom{l}{i}
			   \prod_{n=2}^\infty (1-n^{-s})^{-i} \\
			&= \left( \prod_{n=2}^\infty (1-n^{-s})^{-1}-1 \right)^l
			\text{.}
		\end{align*}
		This last expression is the $l$-th power of the dgf of the series $(\bar{f}(n))_{n\geq 1}$, where $\bar{f}(1)=0$ and $\bar{f}(n)=f(n)$ for $n \geq 2$. Therefore, we find that 
		$a_l$ is the $l$-th Dirichlet power of $\bar{f}$, i.e. $a_l=\bar{f}^{*l}$. An analogous argument gives that 
		$b_l=\bar{g}^{*l}$, where $\bar{g}(1)=0$ and $\bar{g}(n)=g(n)$ for $n \geq 2$.
	\end{remark}

\bibliographystyle{plain}

\begin{thebibliography}{99}
	\bibitem{Cha13}
	M. Chamberland, C. Johnson, A. Nadeau, and B. Wu, 
	Multiplicative Partitions, 
	\textit{Electron. J. Combin.} 
	\textbf{20} (2013), \#P57.	
	
	\bibitem{Fri10}
	J. B. Friedlander and H. Iwaniec
	\textit{Opera de Cribro}
	(2010), American Mathematical Society.
	
	\bibitem{Hen87}
	D. Hensley, 
	The distribution of the number of factors in a factorization, 
	\textit{J. Number Theory} 
	\textbf{26} (1987), 179--191.
	
	\bibitem{Kal30}
	L. Kalm\'{a}r,
	\"{U}ber die mittlere Anzahl der Produktdarstellungen der Zahlen (Erste Mitteilung),
	\textit{Acta Litterarum ac Scientarum, Szeged}
	\textbf{5} (1931), 95--107.
	
	\bibitem{Kno05}	
	A. Knopfmacher and M. E. Mays, 
	A survey of factorisation counting functions, 
	\textit{Int. J. Number Theory} 
	\textbf{1} (2005), 563--581.
	
	\bibitem{Kno06}	
	A. Knopfmacher and M. E. Mays, 
	Ordered and Unordered Factorizations of Integers, 
	\textit{The Mathematica Journal} 
	\textbf{10} (2006), 72--89.
	
	\bibitem{OEIS}
	N. J. A. Sloane et al., The On-Line Encyclopedia of Integer Sequences,
	\url{https://oeis.org}, 2020.
	
	\bibitem{Slo95}
	N. J. A. Sloane and S. Plouffe,  
	\textit{The Encyclopedia of Integer Sequences},
	San Diego, CA: Academic Press, 1995.
	
	\bibitem{Sub04}
	M. V. Subbarao, 
	Product partitions and recursion formulae, 
	\textit{Int. J. Math. Math. Sci.} 
	\textbf{33} (2004),
	1725--1735.
	
	\bibitem{Wie32}
	N. Wiener,
	Tauberian Theorems, 
	\textit{Ann. of Math. (2)}, 
	\textbf{33} (1932), 
	1--100.
	
\end{thebibliography}

\bigskip
\hrule
\bigskip

\noindent 2010 {\it Mathematics Subject Classification}: Primary 11A51;
Secondary 05A17.

\noindent \emph{Keywords:}
number of colored factorizations, number of unordered factorizations, multiplicative partition.

\bigskip
\hrule
\bigskip
		
\end{document}